\documentclass[a4paper,twoside]{amsart}
\usepackage{amssymb}
\usepackage{extarrows}
\usepackage[hmargin={25mm,25mm},vmargin={25mm,25mm}]{geometry}
\bibstyle{plain}
\newtheorem{theorem}{Theorem}
\newtheorem{lemma}[theorem]{Lemma}

\theoremstyle{remark}

\renewcommand{\phi}{\varphi}
\renewcommand{\epsilon}{\varepsilon}

\begin{document}
\title{On the behavior of integrable functions at infinity}
\author{Andrzej Komisarski}
\address{Andrzej Komisarski\\
Department of Probability Theory and Statistics\\Faculty of Mathematics and Computer Science\\
University of \L\'od\'z\\ul.Banacha 22\\90-238 \L\'od\'z\\Poland}
\email{andkom@math.uni.lodz.pl}

%Lesigne 26A06 (26A42), Rum:26A42 (37A45), 26A42 (26A12 37A45), 26A42 (26A03 26A12), 26A42
%+26A42 Integrals of Riemann, Stieltjes and Lebesgue type
% 26A06 One-variable calculus
% 26A03 Foundations: limits and generalizations, elementary topology of the line
% 26A12 Rate of growth of functions, orders of infinity, slowly varying functions [See also 26A48]
%-37A45 Relations with number theory and harmonic analysis [See also 11Kxx]
%+26B15 Integration: length, area, volume [See also 28A75, 51M25]
%+28A25 Integration with respect to measures and other set functions
%+40A05 Convergence and divergence of series and sequences

\begin{abstract}
We investigate the behavior of sequences $(f(c_nx))$ for Lebesgue integrable functions $f:\mathbb R^d\to\mathbb R$.
In particular, we give a~description of classes of multipliers $(c_n)$ and $(d_n)$ such that $f(c_nx)\to0$ or $\sum_{n=1}^\infty|f(d_nx)|<\infty$
for $\lambda$ almost every $x\in\mathbb R^d$.
\end{abstract}

\maketitle

It is well known that if a~series $\sum_{n=1}^\infty a_n$ is convergent, then $a_n\to0$.
It may seem surprising that a~similar result does not hold for integrals. Namely, if $f:\mathbb R\to\mathbb R$ is Lebesgue integrable,
then it is not necessary that $\lim_{x\to\infty}f(x)=0$. Various authors investigated
the behavior of integrable functions at infinity, see e.g. \cite{Lesigne,Mihai,NicPopi,NicPopii,NicPopiii}.

E. Lesigne showed in \cite{Lesigne} that if $f:\mathbb R\to\mathbb R$ is Lebesgue integrable, then for $\lambda$ almost every
$x\in\mathbb R$ one has $f(nx)\to0$. In this paper we generalize Lesigne's investigations in several directions.
One way is to replace the domain of $f$ by the space $\mathbb R^d$ equipped with $d$-dimensional Lebesgue measure $\lambda$.
On the other hand, we want to describe a~possibly large class of multipliers $c_n$ which may be substituted for $n$ in Lesigne's result.
As the first result going in this direction we present the following theorem:

\begin{theorem}\label{szpoz}
Let $d\in\mathbb N$ and let $(c_n)$ be a~sequence of positive numbers such that for some permutation $(c'_n)$ of $(c_n)$
the sequence $(\sqrt[d]n/c_n')$ is bounded.
Assume that $f:\mathbb R^d\to\mathbb R$ is Lebesgue integrable ($\int|f(x)|dx<\infty$). Then for $\lambda$ almost every
$x\in\mathbb R^d$ one has $\sum_{n=1}^\infty|f(c_nx)|<\infty$ (hence $f(c_nx)\to0$).
\end{theorem}

A little comment is necessary to explain the assumption on the sequence $(c_n)$. Theorem \ref{szpoz} would be valid
if we just assumed that $(\sqrt[d]n/c_n)$ is bounded. However, the conclusion of the theorem is permutation invariant,
i.e., if it holds for a~sequence $(c_n)$, then it also holds for any permutation of $(c_n)$.
If any form of the reversal of Theorem \ref{szpoz} should hold true, then its assumptions have to be permutation invariant as well.
Unfortunately, the condition ``$(\sqrt[d]n/c_n)$ is bounded'' is not permutation invariant. For this reason
an~additional sequence $(c'_n)$ (being a~permutation of $(c_n)$) has to be explicitly introduced.

We note that in Theorem \ref{szpoz} we obtain more than we intended. Namely, we get $\sum_{n=1}^\infty|f(c_nx)|<\infty$
instead of $f(c_nx)\to0$. If one wishes to conclude that $f(c_nx)\to0$, then weaker assumptions
on the function $f$ are needed:

\begin{theorem}\label{granpoz}
Let $d\in\mathbb N$ and let $(c_n)$ be a~sequence of positive numbers such that for some permutation $(c'_n)$ of $(c_n)$
the sequence $(\sqrt[d]n/c_n')$ is bounded.
Moreover, let $f:\mathbb R^d\to\mathbb R$ be measurable and such that
for every $\epsilon>0$ one has $\lambda(\{x\in\mathbb R^d:|f(x)|\geq\epsilon\})<\infty$.
Then for $\lambda$ almost every $x\in\mathbb R^d$ one has $f(c_nx)\to0$.
\end{theorem}

In the conclusion of the above theorem we cannot keep the stronger statement $\sum_{n=1}^\infty|f(c_nx)|<\infty$
from Theorem \ref{szpoz}. Indeed, if $f(x)=1/(1+\|x\|)$, then $\sum_{n=1}^\infty|f(nx)|=\infty$ for every $x$.

The next theorem shows that the assumption on the sequence $(c_n)$ in Theorem \ref{szpoz} cannot be weakened.

\begin{theorem}\label{szneg}
Let $d\in\mathbb N$ and let $(c_n)$ be a~sequence of positive numbers such that for every permutation $(c'_n)$ of $(c_n)$
the sequence $(\sqrt[d]n/c_n')$ is unbounded.
Then there exists a~continuous, nonnegative function $f:\mathbb R^d\to\mathbb R$ such that $\int|f(x)|dx<\infty$
and $\sum_{n=1}^\infty|f(c_nx)|=\infty$ for every $x\in\mathbb R^d$.
\end{theorem}

The above theorem may be seen as the inverse of Theorem \ref{szpoz}.
The situation is much more delicate when we try to inverse Theorem \ref{granpoz}.
Consider the following example: Let $(c_n)$ satisfy the assumption of Theorem \ref{granpoz},
for simplicity set $c_n=n$.
Then for any integrable $f$ we have $f(nx)\to0$ for $\lambda$ almost every $x\in\mathbb R^d$.
Now, we define a~sequence $(d_n)$ such that it tends to infinity arbitrarily slowly,
yet $f(d_nx)\to0$ for $\lambda$ almost every $x\in\mathbb R^d$.
It suffices to take $(d_n)$ which is formed by repeating each term of the sequence $(c_n=n)$
finitely many times. Indeed, the convergence of $f(d_nx)$ to zero follows from $f(nx)\to0$.
On the other hand, $(d_n)$ may tend to infinity slowly enough to ensure that $(\sqrt[d]n/d_n')$
is unbounded for every permutation $(d'_n)$ of $(d_n)$.
All this shows that Theorem \ref{granpoz} cannot be fully inversed.
Instead, we show the following theorem:

\begin{theorem}\label{granneg}
Let $d\in\mathbb N$ and let $(c_n)$ be a~sequence of positive numbers such that for every permutation $(c'_n)$ of $(c_n)$
the sequence $(\sqrt[d]n/c_n')$ is unbounded.
Then there exist a~sequence $(b_n)$ of positive numbers and a~continuous, nonnegative, integrable function $f:\mathbb R^d\to\mathbb R$
such that $b_n/c_n\to1$ and $f(b_nx)\not\to0$ for every $x\in\mathbb R^d$.
\end{theorem}
In fact we prove a~bit more: If $c_n\to\infty$, then additionally $\limsup_{n\to\infty}f(b_nx)=\infty$ for every $x\neq0$.

In Theorem \ref{granneg} we claim that if a~sequence $(c_n)$ does not satisfy the assumption of Theorem \ref{granpoz},
then even if it is not ``bad'' itself, it can be slightly modified to a~``bad'' sequence.
On the other hand, each sequence $(c_n)$ with $c_n\to\infty$ can be improved in the following sense:

\begin{theorem}\label{granpozb}
Let $d\in\mathbb N$ and let $(c_n)$ be a~sequence of positive numbers tending to infinity.
There exists a~sequence $(b_n)$ of positive numbers with $b_n/c_n\to1$ such that:
For any measurable $f:\mathbb R^d\to\mathbb R$ satisfying
$\forall_{\epsilon>0}$ $\lambda(\{x\in\mathbb R^d:|f(x)|\geq\epsilon\})<\infty$
one has $f(b_nx)\to0$ for $\lambda$ almost every $x\in\mathbb R^d$.
\end{theorem}

In \cite{Lesigne} Lesigne also investigated the rate of convergence of $(f(nx))$ to zero.
In particular, he showed that for any sequence $(a_n)$ with $0\leq a_n\to\infty$
there exists a~continuous, integrable function $f:\mathbb R\to\mathbb R$
such that $\limsup_{n\to\infty}a_nf(nx)=\infty$ for $\lambda$ almost every $x\in\mathbb R$.
Moreover, if we drop the continuity requirement (we only require integrability of $f$), then we may
obtain $\limsup_{n\to\infty}a_nf(nx)=\infty$ for every $x\in\mathbb R$.
Lesigne asked if we may have both: continuity of $f$ and $\limsup_{n\to\infty}a_nf(nx)=\infty$ for every $x$.
This question has been positively answered by G. Batten in \cite{Batten}. The original Batten's paper is accessible through arXiv,
but (to best our knowledge) has never been published. Here we present a~much shorter proof of Batten's result in $\mathbb R^d$,
based on completely different ideas.

\begin{theorem}\label{odpow}
Let $d\in\mathbb N$ and let a~sequence $(a_n)$ satisfy $0\leq a_n\to\infty$.
There exists a~continuous, nonnegative, integrable function $f:\mathbb R^d\to\mathbb R$
such that $\limsup_{n\to\infty}a_nf(\sqrt[d]nx)=\infty$ for every $x\in\mathbb R^d$.
\end{theorem}

\section*{Proofs}

The following lemma plays a~very important role in the proofs of (almost) all theorems in this paper:

\begin{lemma}\label{lemat}
Let $d>0$ and $a>1$ be real numbers and let $(c_n)$ be a~sequence of positive numbers. The following conditions are equivalent:
\begin{itemize}
\item[(i)] There exists a~permutation $(c'_n)$ of $(c_n)$, such that the sequence $(\sqrt[d]n/c_n')$ is bounded.
\item[(i')] There exists a~(unique) nondecreasing sequence $(c'_n)$ being a~permutation of $(c_n)$
and for this permutation the sequence $(\sqrt[d]n/c_n')$ is bounded.
\item[(ii)] There exists $M>0$, such that $\forall_{t>0}\ \sum_{\{n:\ t\leq c_n<at\}}\frac1{c_n^d}\leq M$.
\item[(iii)] There exists $M'>0$, such that $\forall_{k\in\mathbb Z}\ \frac{|\{n:a^k\leq c_n<a^{k+1}\}|}{a^{kd}}\leq M'$
\end{itemize}
\end{lemma}

\begin{proof}
Clearly (i') implies (i). We will show (i)$\Rightarrow$(iii)$\Rightarrow$(ii)$\Rightarrow$(i')

(i)$\Rightarrow$(iii). Let $L>0$ satisfy $\sqrt[d]n/c_n'\leq L$ for every $n$. If we put $M'=(La)^d$, then for every $k\in\mathbb Z$ we have
\begin{equation*}
\begin{split}
&\frac{|\{n:a^k\leq c_n<a^{k+1}\}|}{a^{kd}}\leq\frac{|\{n:c_n<a^{k+1}\}|}{a^{kd}}=\frac{|\{n:c'_n<a^{k+1}\}|}{a^{kd}}\leq
\frac{|\{n:\sqrt[d]n<La^{k+1}\}|}{a^{kd}}=\\
\qquad&\frac{|\{n:n<(La^{k+1})^d\}|}{a^{kd}}<\frac{(La^{k+1})^d}{a^{kd}}=(La)^d=M'.
\end{split}
\end{equation*}

(iii)$\Rightarrow$(ii). We put $M=2M'$. Let $t>0$. We have $a^{k-1}\leq t<a^k$ for some $k\in\mathbb Z$ and then
\begin{equation*}
\begin{split}
\sum_{\{n:\ t\leq c_n<at\}}&\frac1{c_n^d}\leq\sum_{\{n:\ a^{k-1}\leq c_n<a^k\}}\frac1{c_n^d}+\sum_{\{n:\ a^k\leq c_n<a^{k+1}\}}\frac1{c_n^d}\leq\\
&\frac{|\{n:\ a^{k-1}\leq c_n<a^k\}|}{a^{(k-1)d}}+\frac{|\{n:\ a^k\leq c_n<a^{k+1}\}|}{a^{kd}}\leq M'+M'=M.
\end{split}
\end{equation*}

(ii)$\Rightarrow$(i').
For any $t>0$ we have
\begin{equation*}
\begin{split}
&|\{n:c_n<t\}|=\sum_{k=1}^\infty|\{n:ta^{-k}\leq c_n<ata^{-k}\}|\leq\sum_{k=1}^\infty\sum_{\{n:\ ta^{-k}\leq c_n<ta^{1-k}\}}\frac{(ta^{1-k})^d}{c_n^d}\leq\\
&\sum_{k=1}^\infty(ta^{1-k})^d\cdot M=t^d\cdot\frac{M}{1-1/a^d}.
\end{split}
\end{equation*}
In particular, for every $t>0$ the set $\{n:c_n<t\}$ is finite, hence there exists a~nondecreasing permutation $(c_n')$ of $(c_n)$.
For this permutation we have $|\{n:c_n'<t\}|=|\{n:c_n<t\}|\leq t^d\cdot\frac{M}{1-1/a^d}$.
Since $(c_n')$ is nondecreasing, for every $m\in\mathbb N$ we have:
$$m\leq\inf_{t>c_m'}|\{n:c_n'<t\}|\leq\inf_{t>c_m'}\left(t^d\cdot\frac{M}{1-1/a^d}\right)=c_m'^d\cdot\frac{M}{1-1/a^d},$$
hence $\sqrt[d]m/c_m'\leq\sqrt[d]{\frac{M}{1-1/a^d}}$.
\end{proof}

\begin{proof}[Proof of Theorem \ref{szpoz}]
In the first part of the proof we show that for $\lambda$ almost every $x\in\mathbb R^d$
satisfying $\frac12<\|x\|\leq 1$ we have $\sum_{n=1}^\infty|f(c_nx)|<\infty$.
We define $f_n:\mathbb R^d\to\mathbb R$ by the formula
$$f_n(x)=\frac1{c_n^d}\cdot|f(x)|\cdot\mathbf1_{c_n/2<\|x\|\leq c_n}.$$
Functions $f_n$ are nonnegative and $\sum_{n=1}^\infty f_n(0)=0$. For every $x\neq0$
we use Lemma \ref{lemat} ((i)$\Rightarrow$(ii) with $a=2$ and $t=\|x\|$) to obtain:
$$\sum_{n=1}^\infty f_n(x)=|f(x)|\cdot\sum_{\{n:\ \|x\|\leq c_n<2\|x\|\}}\frac1{c_n^d}\leq |f(x)|\cdot M.$$
It follows that the function series $\sum_{n=1}^\infty f_n(x)$ is convergent and
$\int\sum_{n=1}^\infty f_n(x)dx\leq M\cdot\int|f(x)|dx<\infty$. Hence
\begin{equation*}
\begin{split}
\sum_{n=1}^\infty&\int_{\{x:\frac12<\|x\|\leq1\}}|f(c_nx)|dx=
\sum_{n=1}^\infty\int|f(c_nx)|\cdot\mathbf1_{\frac12<\|x\|\leq1}dx=
\sum_{n=1}^\infty\int|f(x)|\cdot\mathbf1_{c_n/2<\|x\|\leq c_n}\cdot\frac1{c_n^d}dx=\\
&\sum_{n=1}^\infty\int f_n(x)dx=\int\sum_{n=1}^\infty f_n(x)dx<\infty.
\end{split}
\end{equation*}
Thus, the function series $\sum_{n=1}^\infty|f(c_nx)|$ is convergent $\lambda$ almost everywhere
on $\{x\in\mathbb R^d:\frac12<\|x\|\leq1\}$ and the first part of the proof is completed.

Now, for $k\in\mathbb Z$ we consider the function $g_k(x)=f(2^kx)$. Clearly $g_k$ is integrable,
hence, by the first part of the proof, for $\lambda$ almost every $y$ satisfying
$\frac12<\|y\|\leq1$ the series $\sum_{n=1}^\infty|f(c_n2^ky)|=\sum_{n=1}^\infty|g_k(c_ny)|$ converges.
Denoting $x=2^ky$ we obtain that for $\lambda$ almost every $x$ satisfying
$2^{k-1}<\|x\|\leq2^k$ we have $\sum_{n=1}^\infty|f(c_nx)|<\infty$.
This observation completes the proof, because
$\mathbb R^d=\{0\}\cup\bigcup_{k\in\mathbb Z}\{x\in\mathbb R^d:2^{k-1}<\|x\|\leq2^k\}$.
\end{proof}

\begin{proof}[Proof of Theorem \ref{granpoz}]
For $k=1,2,\dots$ we apply Theorem \ref{szpoz} for an~integrable function
$f_k(x)=\mathbf1_{|f(x)|\geq1/k}$. As a~result, we obtain a~set $A_k\subset\mathbb R^d$,
such that $\lambda(A_k)=0$ and for every $x\in\mathbb R^d\setminus A_k$ we have $f_k(c_nx)\to0$ when $n\to\infty$.
Clearly, $\lambda(\bigcup_{k=1}^\infty A_k)=0$.
The convergence $f_k(c_nx)\to0$ implies that the set $\{n:|f(c_nx)|\geq1/k\}$ is finite.
It follows that if $x\in\mathbb R^d\setminus\bigcup_{k=1}^\infty A_k$, then $\forall_{k\in\mathbb N}\ |\{n:|f(c_nx)|\geq1/k\}|<\infty$,
which means $f(c_nx)\to0$.
\end{proof}

\begin{proof}[Proof of Theorem \ref{szneg}]
If $c_n\not\to\infty$, then there exists $c\geq0$ and a~subsequence $(c_{n_i})$ such that $c_{n_i}\to c$.
In this case we can take any $f$ which is strictly positive, integrable and continuous, e.g. $f(x)=1/(1+\|x\|^{d+1})$.
Indeed, if $x\in\mathbb R^d$, then $f(c_{n_i}x)\to f(cx)>0$, hence $\sum_{n=1}^\infty|f(c_nx)|\geq\sum_{i=1}^\infty|f(c_{n_i}x)|=\infty$.
In the remaining part of the proof we assume $c_n\to\infty$.

For $k\in\mathbb Z$ let $A_k=\{n:2^k\leq c_n<2^{k+1}\}$ and $l_k=\sum_{n\in A_k}\frac1{c_n^d}$.
The assumption $c_n\to\infty$ implies that the sets $A_k$ are finite.
Moreover, the sets $A_k$ are pairwise disjoint and $\mathbb N=\bigcup_{k\in\mathbb Z}A_k$.
It follows, that for every $n\in\mathbb N$ there exists the unique $k(n)\in\mathbb Z$ such that $n\in A_{k(n)}$.
By Lemma \ref{lemat} ($\neg$(i)$\Rightarrow\neg$(iii) with $a=2$) and by the inequality $l_k\geq\frac{|A_k|}{2^{(k+1)d}}$
we obtain that the set $\{l_k:k\in\mathbb Z\}$ is unbounded. We take a~sequence $(k_i)$ such that $k_i$'s are pairwise different
and $l_{k_i}\geq i$ for every $i$. We define nonnegative numbers $(r_k)_{k\in\mathbb Z}$ by the formula
$$r_k=\begin{cases}
\frac1{i^2|A_{k_i}|}&\text{if }k=k_i,\\
0&\text{if }k\neq k_i\text{ for every }i.
\end{cases}$$
(note that $l_{k_i}>0$ implies $A_{k_i}\neq\emptyset$).
Then
$$\sum_{m=1}^\infty r_{k(m)}=\sum_{k\in\mathbb Z}\sum_{m\in A_k}r_k=\sum_{k\in\mathbb Z}r_k|A_k|=
\sum_{i=1}^\infty r_{k_i}|A_{k_i}|=\sum_{i=1}^\infty\frac1{i^2}<\infty$$
and $\sum_{k\in\mathbb Z}r_k|A_k|l_k=\sum_{i=1}^\infty r_{k_i}|A_{k_i}|l_{k_i}\geq\sum_{i=1}^\infty\frac1i=\infty$.

Let $g:\mathbb R^d\to\mathbb R$ be any bounded, strictly positive, integrable and continuous function, 
such that $g(x)$ is a~nonincreasing function of $\|x\|$ (e.g., $g(x)=1/(1+\|x\|^{d+1})$).
We define
$$f(x)=\sum_{m=1}^\infty\frac{r_{k(m)}}{c_m^d}\cdot g\left(\frac x{c_m}\right).$$
Note that the above function series converges uniformly, because $g$ is bounded, $c_m\to\infty$
and $\sum_{m=1}^\infty r_{k(m)}<\infty$. In particular $f$ is continuous. Clearly $f$ is positive.
Moreover,
$$\sum_{m=1}^\infty\int\frac{r_{k(m)}}{c_m^d}\cdot g\left(\frac x{c_m}\right)dx=
\sum_{m=1}^\infty\int r_{k(m)}\cdot g(x)dx=\int g(x)dx\cdot\sum_{m=1}^\infty r_{k(m)}<\infty,$$
hence $f$ is integrable.

If $x=0$, then $\sum_{n=1}^\infty|f(c_nx)|=\sum_{n=1}^\infty|f(0)|=\infty$, because $f(0)>0$. For $x\neq0$ we have
\begin{equation*}
\begin{split}
\sum_{n=1}^\infty&|f(c_nx)|=\sum_{k\in\mathbb Z}\sum_{n\in A_k}f(c_nx)\geq
\sum_{k\in\mathbb Z}\sum_{n\in A_k}\sum_{m\in A_k}\frac{r_{k(m)}}{c_m^d}\cdot g\left(\frac{c_n}{c_m}\cdot x\right)\geq\\
&\sum_{k\in\mathbb Z}r_k\sum_{n\in A_k}\sum_{m\in A_k}\frac1{c_m^d}\cdot g(2x)=
g(2x)\cdot\sum_{k\in\mathbb Z}r_k|A_k|l_k=\infty
\end{split}
\end{equation*}
(we used the following observation: if $m,n\in A_k$, then $\frac{c_n}{c_m}<2$).
\end{proof}

The proof of Theorem \ref{granneg} is presented at the end of the paper. It is the hardest proof and it uses
some ideas presented in the proof of Theorem \ref{odpow}. For this reasons leaving it for the end is a~good idea.

\begin{proof}[Proof of Theorem \ref{granpozb}]
Let $(b_n)=(\lceil c_n\rceil)$. Then all the terms of $(b_n)$ are in $\mathbb N$.
The assumption $c_n\to\infty$ assures that for every $k\in\mathbb N$ the set $\{n:b_n=k\}$ is finite.
By Theorem \ref{granpoz} we have $f(kx)\to0$ for $\lambda$ almost every $x\in\mathbb R^d$.
Thus $f(b_nx)\to0$ for $\lambda$ almost every $x\in\mathbb R^d$.
Moreover, $c_n\to\infty$ implies $\frac{b_n}{c_n}=\frac{\lceil c_n\rceil}{c_n}\to1$.
\end{proof}

\begin{proof}[Proof of Theorem \ref{odpow}]
It is enough to construct a~continuous, nonnegative, integrable function $\widetilde f:[0,\infty)\to\mathbb R$,
such that $\limsup_{n\to\infty}a_n\widetilde f(nx)=\infty$ for every $x\in[0,\infty)$.
Then we define $f:\mathbb R^d\to\mathbb R$ by $f(x)=\widetilde f(\|x\|^d)$.
Clearly, $f$ is continuous, nonnegative and
$\limsup_{n\to\infty}a_nf(\sqrt[d]nx)=\limsup_{n\to\infty}a_n\widetilde f(n\|x\|^d)=\infty$
for every $x\in\mathbb R^d$. Moreover,
$$\int f(x)dx=\int\widetilde f(\|x\|^d)dx=S_d\cdot\int_{r=0}^\infty\widetilde f(r^d)r^{d-1}dr=\frac{S_d}d\int_{y=0}^\infty\widetilde f(y)dy<\infty$$
(here $S_d$ is $d-1$-dimensional measure of the unit sphere in $\mathbb R^d$).

For $k\in\mathbb N$ let $t_k>0$ be such that  $n\geq t_k\Rightarrow a_n\geq k^4$ for every $n\in\mathbb N$.
Let $h:\mathbb R\to\mathbb R$ be any continuous, bounded, nonnegative, integrable function satisfying $h|_{[0,1]}\geq 1$.
We define $\widetilde f:[0,\infty)\to\mathbb R$ as follows:
$$\widetilde f(x)=h(x)+\sum_{l=1}^\infty\frac{h(\frac xl-t_l)}{l^3}.$$
Function $\widetilde f$ is nonnegative and continuous (the series converges uniformly). It is also integrable:
\begin{equation*}
\begin{split}
\int_0^\infty\widetilde f(x)dx=&\int_0^\infty h(x)dx+\sum_{l=1}^\infty\int_0^\infty\frac{h(\frac xl-t_l)}{l^3}dx=
\int_0^\infty h(x)dx+\sum_{l=1}^\infty\int_0^\infty\frac{h(x-t_l)}{l^2}dx\leq\\
&\int h(x)dx+\sum_{l=1}^\infty\int\frac{h(x)}{l^2}dx=\int h(x)dx\cdot\left(1+\sum_{l=1}^\infty\frac1{l^2}\right)<\infty.
\end{split}
\end{equation*}
If $x=0$, then $\limsup_{n\to\infty}a_n\widetilde f(nx)\geq\limsup_{n\to\infty}a_nh(nx)=\limsup_{n\to\infty}a_nh(0)\geq\limsup_{n\to\infty}a_n=\infty$.
Let $x>0$. Then for every $k\in\mathbb N$ satisfying $k>x$ we have $0<\frac xk<1$ and there exists
$n_k\in\mathbb N$ such that $n_k\cdot\frac xk\in[t_k,t_k+1]$, i.e., $\frac{n_kx}k-t_k\in[0,1]$. In particular, $n_k\geq t_k\cdot\frac{k}x>t_k$, hence $a_{n_k}\geq k^4$.
It follows that
$$a_{n_k}\widetilde f(n_kx)\geq a_{n_k}\cdot\frac1{k^3}\cdot h\left(\frac{n_kx}k-t_k\right)\geq k^4\cdot\frac1{k^3}\cdot1=k,$$
thus $\limsup_{n\to\infty}a_n\widetilde f(nx)\geq\limsup_{k\to\infty}a_{n_k}\widetilde f(n_kx)\geq\limsup_{k\to\infty}k=\infty$.
\end{proof}

The following technical lemma is helpful to perform an~inductive construction in the proof of Theorem~\ref{granneg}.

\begin{lemma}\label{lematgranneg}
Let $(c_n)$ be a~sequence of positive numbers such that $c_n\to\infty$ and for every permutation $(c'_n)$ of $(c_n)$
the sequence $(n/c_n')$ is unbounded.
Then for every $a>1$, $\epsilon>0$, $S>0$, $l\in\mathbb Z$ and $M\in\mathbb N\cup\{0\}$ there exist
$T>S$, $\mathbb N\ni N>M$, $b_{M+1},b_{M+2},\dots,b_N>0$ and a~continuous, integrable, nonnegative
function $g:[0,\infty)\to\mathbb R$ satisfying $\frac1a\leq\frac{b_n}{c_n}\leq a$ for $n=M+1,M+2,\dots,N$,
$\int_0^\infty g(x)dx<\epsilon$, $g|_{[0,\infty)\setminus[S,T]}=0$
and $\forall_{x\in[a^{l-1},a^l]}$ $\max_{M<n\leq N}g(b_nx)\geq1$.
\end{lemma}

\begin{proof}
For $k\in\mathbb Z$ let $A_k=\{n:a^k\leq c_n<a^{k+1}\}$. According to Lemma \ref{lemat}
($\neg$(i)$\Rightarrow\neg$(iii)) there exists a~sequence $(k_i)$ such that $\frac{|A_{k_i}|}{a^{k_i}}\to\infty$.
We can assume that $A_{k_i}\neq\emptyset$ and $k_i>1-l+\log_aS$ and $k_i>\max\{\log_ac_n:n\leq M\}$ for every $i$.
The last inequality ensures that for every $n$ if $n\in A_{k_i}$, then $n>M$.
We consider a~term $a^{k_i+l}(1-a^{-1/|A_{k_i}|})$ and its limit when $i\to\infty$:
$$\lim_{i\to\infty}a^{k_i+l}(1-a^{-1/|A_{k_i}|})=\lim_{i\to\infty}a^l\cdot\frac{a^{k_i}}{|A_{k_i}|}\cdot\frac{1-a^{-1/|A_{k_i}|}}{0-(-1/|A_{k_i}|)}
=a^l\cdot0\cdot\log_e a=0.$$
It follows that we can choose $K\in\{k_i:i\in\mathbb N\}$ satisfying $a^{K+l}(1-a^{-1/|A_K|})<\epsilon$.
We put $N=\max A_K$. Then $A_K\subset\{M+1,M+2,\dots,N\}$.

We define $b_{M+1},b_{M+2},\dots,b_N$:
If $n\in\{M+1,\dots,N\}\setminus A_K$, then we put $b_n=c_n$. The remaining $b_n$'s (with $n\in A_K$)
are chosen in any way satisfying $\{b_n:n\in A_K\}=\{a^{K+\frac j{|A_K|}}:j=0,1,\dots,|A_K|-1\}$.
If $n\in\{M+1,\dots,N\}\setminus A_K$, then $\frac1a\leq1=\frac{b_n}{c_n}\leq a$.
If $n\in A_K$, then both $b_n$ and $c_n$ are in $[a^K,a^{K+1})$, hence $\frac1a\leq\frac{b_n}{c_n}\leq a$.

We choose any $T>a^{K+l}$. The inequality $K>1-l+\log_aS$ implies $a^{K+l-\frac1{|A_K|}}\geq a^{K+l-1}>S$.
Hence $[a^{K+l-\frac1{|A_K|}},a^{K+l}]\subset(S,T)$. We also have
$\int_0^\infty\mathbf1_{[a^{K+l-\frac1{|A_K|}},a^{K+l}]}dx=\lambda([a^{K+l-\frac1{|A_K|}},a^{K+l}])=a^{K+l}(1-a^{-1/|A_K|})<\epsilon$.
All these observations show that there exists a~nonnegative, continuous function $g:[0,\infty)\to\mathbb R$ such that
$g$ equals $0$ outside $[S,T]$, $g$ equals $1$ on $[a^{K+l-\frac1{|A_K|}},a^{K+l}]$ and $\int_0^\infty g(x)dx<\epsilon$.

It remains to check that $\forall_{x\in[a^{l-1},a^l]}$ $\max_{M<n\leq N}g(b_nx)\geq1$.
We have
$$[a^{l-1},a^l]=\bigcup_{j=0}^{|A_K|-1}\left[a^{l-\frac{j+1}{|A_K|}},a^{l-\frac j{|A_K|}}\right]
=\bigcup_{n\in A_K}\left[\frac{a^{K+l-\frac1{|A_K|}}}{b_n},\frac{a^{K+l}}{b_n}\right].$$
It follows, that if $x\in[a^{l-1},a^l]$, then $b_{n_0}x\in[a^{K+l-\frac1{|A_K|}},a^{K+l}]$ for some $n_0\in A_K$.
Consequently, $\max_{M<n\leq N}g(b_nx)\geq g(b_{n_0}x)=1$.
\end{proof}

\begin{proof}[Proof of Theorem \ref{granneg}]
If $c_n\not\to\infty$, then there exists $c\geq0$ and a~subsequence $(c_{n_i})$ such that $c_{n_i}\to c$.
In this case we can take any $f$ which is strictly positive, integrable and continuous, e.g. $f(x)=1/(1+\|x\|^{d+1})$
and $(b_n)=(c_n)$. Indeed, if $x\in\mathbb R^d$, then $f(c_{n_i}x)\to f(cx)>0$, hence $f(b_nx)=f(c_nx)\not\to0$.
In the remaining part of the proof we assume $c_n\to\infty$.

Let $(\widetilde c_n)=(c_n^d)$. Then $\widetilde c_n\to\infty$
and for every permutation $(\widetilde c'_n)$ of $(\widetilde c_n)$
the sequence $(n/\widetilde c_n')$ is unbounded.
To finish the proof it is enough to construct a~continuous, nonnegative, integrable function $\widetilde f:[0,\infty)\to\mathbb R$
and a~sequence $(\widetilde b_n)$ such that $\frac{\widetilde b_n}{\widetilde c_n}\to1$
and $\widetilde f(\widetilde b_nx)\not\to0$ for every $x\in[0,\infty)$.
Then the function $f:\mathbb R^d\to\mathbb R$ defined by $f(x)=\widetilde f(\|x\|^d)$
is continuous, nonnegative and integrable (see the beginning of the proof of Theorem \ref{odpow}).
For $(b_n)=(\sqrt[d]{\widetilde b_n})$ we have $\frac{b_n}{c_n}=\sqrt[d]{\frac{\widetilde b_n}{\widetilde c_n}}\to1$
and $f(b_nx)=\widetilde f(\widetilde b_n\|x\|^d)\not\to0$ for every $x\in\mathbb R^d$.

We fix two sequences: $(a_i)$ and $(l_i)$ such that $a_i>1$ and $l_i\in\mathbb Z$ for every $i\in\mathbb N$,
$a_i\to 1$ and every $x>0$ is an~element of infinitely many of the intervals $[a_i^{l_i-1},a_i^{l_i}]$.
One may put for example
\begin{equation*}
\begin{split}
a_i=&1+\frac1k\\
l_i=&i-\frac{(k+1)^3+k^3-1}2
\end{split}
\qquad\text{for } k^3\leq i<(k+1)^3,\quad k\in\mathbb N
\end{equation*}
(it is easy to compute that for such $(a_i)$ and $(l_i)$ one has $\bigcup_{i=k^3}^{(k+1)^3-1}[a_i^{l_i-1},a_i^{l_i}]\supset[2^{-k},2^k]$).

We construct the function $\widetilde f$ and the sequence $(\widetilde b_n)$ piecewise, by induction. In each step we apply Lemma \ref{lematgranneg}
to obtain the next part of the function $\widetilde f$ and the next part of the sequence $(\widetilde b_n)$.
More precisely, in the $i$-th step of the induction we define $\widetilde f$ on an~interval $[S_i,T_i]$ and $\widetilde b_n$'s with $n=M_i+1,\dots,N_i$.
At the beginning no $\widetilde b_n$'s are defined, so we put $M_1=0$. We choose $S_1$ arbitrarily, e.g. $S_1=1$.
Then we apply Lemma \ref{lematgranneg} with $a=a_1$, $l=l_1$, $M=M_1$, $S=S_1$ and $\epsilon=1/4$.
As a~result we obtain $N_1=N$, $T_1=T$, function $g_1=g:[0,\infty)\to\mathbb R$ such that $g_1$ is zero outside $[S_1,T_1]$
and $\widetilde b_n$'s for $n=M_1+1,\dots,N_1$. We repeat this procedure infinitely many times. In the $i$-th step
we apply Lemma \ref{lematgranneg} with $a=a_i$, $l=l_i$, $M=M_i=N_{i-1}$, $S=S_i=T_{i-1}+1$ and $\epsilon=1/4^i$.
As a~result we obtain $N_i=N$, $T_i=T$, function $g_i=g:[0,\infty)\to\mathbb R$ such that $g_i$ is zero outside $[S_i,T_i]$
and $\widetilde b_n$'s for $n=M_i+1,\dots,N_i$.

The whole sequence $(\widetilde b_n)$ satisfies $\frac1{a_i}\leq\frac{\widetilde b_n}{\widetilde c_n}\leq a_i$
for $M_i<n\leq N_i$, which (together with $a_i\to1$) implies $\frac{\widetilde b_n}{\widetilde c_n}\to1$.
Let
$$\widetilde f(x)=h(x)+\sum_{i=1}^\infty 2^ig_i(x),$$
where $h:[0,\infty)\to\mathbb R$ is an~arbitrary continuous, positive and integrable function.
Function $\widetilde f$ is nonnegative, continuous (the series converges almost uniformly)
and integrable ($\int_0^\infty\widetilde f(x)dx<\int_0^\infty h(x)dx+\sum_{i=1}^\infty 2^i/4^i<\infty$).

Finally, let $x\in[0,\infty)$.
If $x=0$, then $\widetilde f(\widetilde b_nx)\geq h(0)>0$, hence $\widetilde f(\widetilde b_nx)\not\to0$.
If $x>0$, then there exists an~increasing sequence $(i_j)$ satisfying $x\in[a_{i_j}^{l_{i_j}-1},a_{i_j}^{l_{i_j}}]$
and we have
$$\limsup_{n\to\infty}\widetilde f(\widetilde b_nx)=
\limsup_{i\to\infty}\max_{M_i<n\leq N_i}\widetilde f(\widetilde b_nx)\geq
\limsup_{j\to\infty}\max_{M_{i_j}<n\leq N_{i_j}}2^{i_j}g_{i_j}(\widetilde b_nx)\geq
\limsup_{j\to\infty}2^{i_j}=\infty.$$
\end{proof}

\end{document}